\documentclass[11pt, oneside,reqno]{amsart}
\usepackage{geometry}                
\usepackage{graphicx}
\usepackage{subfig}
\usepackage{amssymb}
\usepackage{epstopdf}
\usepackage{mathrsfs,amssymb}
\usepackage{graphicx}
\usepackage{amscd}
\usepackage{amsmath}
\usepackage{amssymb}
\usepackage{amsthm}
\usepackage{amsfonts}
\usepackage{fancyhdr}
\usepackage{amsxtra}
\usepackage{tikz}
\usepackage{setspace}
\usepackage{enumitem}
\usepackage{amsrefs}
\newtheorem{theorem}{Theorem}
\newtheorem{lemma}{Lemma}
\newtheorem{corollary}{Corollary}
\newtheorem{proposition}{Proposition}

\newtheorem{remark}{Remark}

\def\bt{\begin{theorem}}
\def\et{\end{theorem}}
\def\bp{\begin{proposition}}
\def\ep{\end{proposition}}
\def\bl{\begin{lemma}}
\def\el{\end{lemma}}

\numberwithin{equation}{section}
\numberwithin{theorem}{section}
\numberwithin{proposition}{section}
\numberwithin{lemma}{section}
\numberwithin{corollary}{section}
\numberwithin{remark}{section}
\numberwithin{definition}{section}
\numberwithin{claim}{section}

\begin{document}
\setcounter{tocdepth}{1}

\title[On the solutions of Navier-Stokes equations for turbulent channel flows in a particular function class]{On the solutions of Navier-Stokes equations for turbulent channel flows in a particular function class} 

\date{\today}
\author{J. Tian$^{1}$}
\address{$^1$ Department of Mathematics\\
Towson University\\
Towson, MD 21252, USA}
\address{$^2$Department of Mathematics\\
	Texas A\&M University\\ College Statoin, TX, 77843}
\address{$\dagger$ corresponding author}
\author{B. Zhang$^{2}$}
\email[J. Tian$^\dagger$]{jtian@towson.edu}
\email[B. Zhang]{shanby84@gmail.com }

\begin{abstract}
In this paper, we continue the discussion as done in \cite{CTZ15} on turbulent channel flow described by the Navier-Stokes model and the Navier-Stokes-alpha model. We study the non-stationary solutions for the Navier-Stokes equations and Navier-Stokes-$\alpha$ model having particular function forms. In particular, the term of sum of pressure and potential
can be shown to be harmonic in the space variable.
\end{abstract}

\maketitle
\section{Introduction}
Turbulence is a fluid regime with the characteristics of being unsteady, irregular, seemingly random and chaotic \cite{gnedin2016star}. It can be used to model the weather, ocean currents, water 
flows in a pipe and
air 
flows around the aircraft wing. Studying turbulent fluid flows involves some of the most difficult and fundamental problems in classical physics, and is also of tremendous practical importance.
The Navier-Stokes equaitons (NSE) have been widely used to describe the motion of turbulent fluid flows. However, solving NSE using the direct numerical simulation method for turbulent flows is extremely difficult, since accurate simulation of turbulent flows should account for the interactions of a wide range of scales which leads to high computational costs.
Turbulence modeling could
provide qualitative and in some cases quantitative measures for
many applications \cite{zhao2004dynamic}. There are several types of turbulence modeling methods, for example Reynolds Average Navier-Stokes (RANS) and Large Eddy Simulation (LES). As done in \cite{A02} and \cite{A04}, we accept that the NS-alpha (also called viscous Camassa-Holm equations or Lagrangian averaged Navier-Stokes equations) is a well-suited mathematical model for the dynamics of appropriately averaged turbulent fluid flows. 
Moreover, in 2003, Guermond et al. mentioned NS-alpha as one of the LES models \cite{JJS03}. The possibility that the NS-alpha is an averaged version of the NSE, first considered in \cite{CFHO98} and \cite{CFHO99}, was entailed by several auspicious facts. Namely, the NS-alpha analogue of the Poiseuille, resp, Hagen, solution in a channel, resp, a pipe, displays both the classical Von K$\acute{a}$rm$\acute{a}$n  and the recent Barenblatt-Chorin laws (\cite{H99}). In addition, the NS-alpha analogue of the Hagen solution, when suitably calibrated, yields good approximations to many experimental data \cite{CFHO99}. Therefore, continuing the study of NS-alpha is extremely useful and important in the aspects of both mathematical theories and down-to-earth applications. 


To understand the connection for fluid flows described by NSE and by NS-alpha, in \cite{CTZ15} we
used a simple Reynolds type averaging. We also restrict our consideration to channel flows having special
function forms prescribed as a function class called $\mathcal{P}$. This function class $\mathcal{P}$ was inspired by the concept of regular part of the weak attractor of the 3D NSE (\cite{CR87}, \cite{CRR10}) as well as by that of the sigma weak attractor introduced in \cite{BFL}. This leads us to consider the solution for the channel flow whose averaged form has both the second and third velocity component to
be zero. This will be our assumption for the discussion done in current work. Starting from there, a physical model for the wall roughness of the channel is subsequently provided to show that the NS-$\alpha$ model occurs naturally as the fluid flows. Moreover, by restricting to consider functions from $\mathcal{P}$, a rational explanation was given to facilitate the understanding of why, as the Reynolds number increases, the fluid becomes in favor of the NS-$\alpha$ model instead of the NSE. The class $\mathcal{P}$ was composed by five assumptions, each assumption plays an unique and important role. In this paper, we conduct a deep study of the properties of solutions in this particular function class. We first try to find the explicit formula of the non-stationary solutions for NSE and NS-$\alpha$. This particular solution has the form which only the first velocity component is nonzero. From there, we can recover the classic Poiseuille flow. Moreover, we proved the symmetric property of the integration form of this Poiseuille flow. Explicit and detailed energy estimate of the velocity field in this class $\mathcal{P}$ is presented and is of use to show the connection between $\mathcal{P}$ and the weak global attractor of the equations. Moreover, for the term of sum of pressure and potential, we also proved that it is actually harmonic in the space variable. Studying the properties of the term of sum of pressure and potential, we find an alternative weaker condition of the last assumption in class $\mathcal{P}$. Therefore, we have found an optimal choice of the class $\mathcal{P}$.

The paper is organized as follows. Section 2 gives elementary results on the Navier-Stokes equations and the Navier-Stokes-$\alpha$ model as well as the definition of the class $\mathcal{P}$. In section 3, we solve the channel flows whose velocity fields have a special form. Section 4 contains the energy estimate for solutions in $\mathcal{P}$. In section 5, we discuss the harmonicity of the term of sum of pressure and potential in the space variable and its consequences. The last section contains some basic inequalities together with their proofs.

\section{Preliminaries}
\subsection{Mathematical backgrounds}
Throughout, we consider an incompressible viscous fluid in an immobile region $\mathcal{O}\subset \mathbb{R}^3$ subjected to a potential body force $F=-\nabla \Phi$, with a time independent potential $\Phi=\Phi(x)\in C^{\infty}(\mathcal{O})$. The velocity field of such flows,
\begin{align}
\label{nse_form}
u=u(x,t)=(u_1(x,t),u_2(x,t),u_3(x,t)), x=(x_1,x_2,x_3)\in \mathcal{O}
\end{align}
satisfies the NSE,
\begin{align}
\label{nse}
\frac{\partial}{\partial t}u+(u\cdot \nabla)u=\nu \Delta u-\nabla P, \hspace{.2 in} \nabla\cdot u=0,
\end{align}
where $P:=p+\Phi$, $t$ denotes the time, $\nu>0$ the kinematic viscosity, and $p=p(x,t)$ the pressure.

The NS-$\alpha$ are 
\begin{align}
\label{che}
\frac{\partial}{\partial t}v+\left(u\cdot\nabla\right)v+\sum_{j=1}^{3}v_j\nabla u_j=\nu \Delta v-\nabla Q, \hspace{.2 in} \nabla \cdot u=0,
\end{align}
where
\begin{align}
\label{vche_nse}
v=(v_1,v_2,v_3)&=(1-\alpha^2\Delta)u \nonumber \\
&=\left((1-\alpha^2\Delta)u_1,(1-\alpha^2\Delta)u_2,  (1-\alpha^2\Delta)u_3\right), 
\end{align}
and $Q$ in (\ref{che}) (like $P$ in (\ref{nse})) may depend on the time $t$.

The following boundary conditions, for both the NSE (\ref{nse}) and the NS-$\alpha$ (\ref{che}),
\begin{align}
\label{no_slip_bc}
u(x,t)=0, \hspace{.1 in}  \text{ for } x\in \partial \mathcal{O}:=\text{boundary of } \mathcal{O},
\end{align}
are imposed, since the fluid we consider is viscous.

One can see that if $\alpha=0$, the NS-$\alpha$ (\ref{che}) reduce to NSE (\ref{nse}), so that (\ref{che}) is also referred as an $\alpha$-model of (\ref{nse}).

In the case of a channel flow, that is, $\mathcal{O}=\mathbb{R}\times\mathbb{R}\times[x_3^{(l)},x_3^{(u)}]$, where $h:=x_3^{(u)}-x_3^{(l)}>0$ is the ``height" of the channel, we recall that a vector of the form
\begin{align}
\label{vel_par_form}
(U(x_3),0,0)
\end{align}
is a $stationary$ (i.e., time independent) solution of the NSE (\ref{nse}) if and only if 
\begin{align}
\label{nse_par_form}
U(x_3)=b\left(1-\frac{(x_3-\frac{x_3^{(u)}+x_3^{(l)}}{2})^2}{(h/2)^2}\right), \hspace{.1 in} x_3\in[x_3^{(l)},x_3^{(u)}],
\end{align}
where $b$ is a constant velocity;
respectively, $(U(x_3),0,0)$ is a $stationary$ (i.e., time independent) solution of the NS-$\alpha$ (\ref{che}) if and only if,
\begin{align}
\label{vche_par_form}
U(x_3)=a_1\left(1-\frac{\cosh\left((x_3-\frac{x_3^{(u)}+x_3^{(l)}}{2})/{\alpha}\right)}{\cosh h/(2\alpha)}\right)+a_2\left(1-\frac{(x_3-\frac{x_3^{(u)}+x_3^{(l)}}{2})^2}{(h/2)^2}\right), 
\end{align}
for $x_3\in[x_3^{(l)},x_3^{(u)}]$,
where $a_1, a_2$ are constant velocities (cf. formula (9.6) in \cite{CFHO99}). Above, $\cosh(x)=\frac{e^{x}+e^{-x}}{2}$ is the hyperbolic cosine function.

To simplify our notation, we will assume that $x_3^{(l)}=0$ and $x_3^{(u)}=h$.
\subsection {The class $\mathcal{P}$}
We first recall the class $\mathcal{P}$ defined in \cite{CTZ15}.

By definition, a function $u(x,t)$ belongs to class $ \mathcal{P}$ if it satisfies $(\bf{A}.1)-(\bf{A}.5)$,

$(\bf{A}.1)$ $u(x,t)\in C^{\infty}(\mathcal{O}\times \mathbb{R})$.

$(\bf{A}.2)$ $u(x,t)$ is periodic in $x_1$ and $x_2$, with periods $\Pi_1$ and $\Pi_2$, respectively; i.e., 
\begin{align}
\label{periodicity_u}
u(x_1+\Pi_1,x_2,x_3,t)=u(x_1,x_2,x_3,t), \hspace{.1 in} u(x_1,x_2+\Pi_2,x_3,t)=u(x_1,x_2,x_3,t).
\end{align}

\begin{remark}
\label{rk_pressure_dp}
We remark that from $(\ref{nse})$ and $(\bf{A}.2)$, it follows,
\begin{align}
\label{pressure_difference}
\left\{\begin{matrix}
P(x_1+\Pi_1,x_2,x_3,t)-P(x_1,x_2,x_3,t)=:p_1(t),\\
P(x_1,x_2+\Pi_2,x_3,t)-P(x_1,x_2,x_3,t)=:p_2(t),
\end{matrix}\right.
\end{align}
\end{remark}

$(\bf{A}.3)$ $u(x,t)$ exists for all $t\in \mathbb{R}$, and has bounded energy per mass, i.e.,
\begin{align}
\label{bdd}
 \int_{0}^{\Pi_1} \int_{0}^{\Pi_2}\int_{0}^{h}u(x,t)\cdot u(x,t) dx<\infty, \forall t\in \mathbb{R}. 
\end{align}

$(\bf{A}.4)$ there exists a constant $ \bar{p}<\infty$ for which,
\begin{align*}
&0<-p_1(t)\leq \bar{p}\\
&|p_2(t)|\leq \bar{p}  
\end{align*}
for all $t \in \mathbb{R}$, where $p_1(t)$ and $p_2(t)$ are defined in (\ref{pressure_difference}).


 $(\bf{A}.5)$ $P=P(x,t)$ is bounded in $x_2$ direction, i.e.,
\begin{align*}
\sup_{x_2 \in \mathbb{R}}P(x_1,x_2,x_3,t)<\infty, \forall x_1,x_3,t\in \mathbb{R}.
\end{align*}

Notice that, as done in \cite{CTZ15}, one can show 
\begin{align*}
p_2(t)\equiv 0.
\end{align*}


As mentioned in the introduction, in the discussion of current paper, we additionally assume that $u_3(x,t)\equiv 0$.

The reality condition on $u$ becomes, when viewed in the Fourier space,
\begin{align}
\label{reality}
\hat{u}^{*}(t;k_1,k_2,k)=\hat{u}(t;-k_1,-k_2,k);
\end{align}
and due to $u_3=0$, the divergence free condition in $(\ref{nse})$ reduces to 
\begin{align}
\label{simp_div_free}
\frac{2\pi k_1}{\Pi_1}\hat{u}_1+\frac{2\pi k_2}{\Pi_2}\hat{u}_2=0.
\end{align}
Using $u_3=0$, we see NSE (\ref{nse}) become
\begin{align}
\label{nse_simple}
\left\{\begin{matrix}
\frac{\partial}{\partial t}u_1+u_1 \frac{\partial}{\partial x_1}u_1+u_2\frac{\partial}{\partial x_2}u_1-\nu (\frac{\partial^2}{\partial x_1^2}+\frac{\partial^2}{\partial x_2^2}+\frac{\partial^2}{\partial x_3^2})u_1=-\frac{\partial}{\partial x_1}P\\ 
\frac{\partial}{\partial t}u_2+u_1 \frac{\partial}{\partial x_1}u_2+u_2\frac{\partial}{\partial x_2}u_2-\nu (\frac{\partial^2}{\partial x_1^2}+\frac{\partial^2}{\partial x_2^2}+\frac{\partial^2}{\partial x_3^2})u_2=-\frac{\partial}{\partial x_2}P\\
0=-\frac{\partial}{\partial x_3}P\\
\frac{\partial}{\partial x_1}u_1+\frac{\partial}{\partial x_2}u_2=0.
\end{matrix}\right.
\end{align}

If one defines the following Reynolds type average for any given scalar function $\phi=\phi(x)$, namely,
\begin{align}
\label{avg_q}
<\phi>(x_3):=\frac{1}{\Pi_1 \Pi_2} \int_{0}^{\Pi_1} \int_{0}^{\Pi_2} \phi dx_2dx_1,
\end{align}

then, one can show

\begin{proposition}[see \cite{CTZ15}]
\label{zero_avg_2nd}
For all $u(x,t)\in \mathcal{P}$, we have
\begin{align}
\label{avg_zero}
<u_2(t)>(x_3)\equiv 0, 
\end{align}
for all $ x_3 \in [0,h], t\in \mathbb{R}$.

Therefore, the averaged velocity field takes the form
\begin{align}
\label{averaged_velocity}
<u(t)>(x_3)=\begin{pmatrix}
<u_1(t)>(x_3)\\
<u_2(t)>(x_3)\\
<u_3(t)>(x_3)
\end{pmatrix}=\begin{pmatrix}
<u_1(t)>(x_3)\\
0\\
0
\end{pmatrix}.
\end{align}
\end{proposition}

The following kernel representation of the averaged velocity component $<u_1(t)>(x_3)$ is also given in \cite{CTZ15}.

\begin{proposition}
\label{repr_by_kernel}
The following relation holds for $u(x,t)\in \mathcal{P}$
\begin{align}
\label{kernel_repr}
<u_1(t)>(x_3)=\int_{-\infty}^{t} K(x_3,t-\tau)p_1(\tau)d\tau,
\end{align}
where, the kernel function $K(x,t)$ is defined by the series,
\begin{align}
\label{kernel_K}
K(x,t)=\sum_{k=1}^{\infty}\frac{2\left((-1)^k-1\right)}{\Pi_1 k \pi}e^{-\nu(\frac{\pi k}{h})^2t}\sin\frac{\pi k x}{h}.
\end{align}
\end{proposition}

\section{Channel flows with velocity field of a particular form}

As shown in the Proposition \ref{zero_avg_2nd}, the averaged velocity field in the solution of the NSE (\ref{nse}) has a special form, namely, both the second and the third components vanish (see (\ref{averaged_velocity})), thus, it is worth to consider the solutions with this form of the NSE and NS-$\alpha$ in a more general setting. See also (\ref{nse_par_form}) and (\ref{vche_par_form}) for 
time independent solutions of this form.

\subsection{The NSE case}

Consider the NSE (\ref{nse}) with the following form of solution, 
\begin{align}\label{assump_1}
	u=(U(x,t),0,0),
\end{align}
for $0\leq x_3\leq h$, satisfying the assumptions $(\bf{A}.1)$, $(\bf{A}.2)$ and $(\bf{A}.3)$.

\subsubsection{Simple form}
Using the form (\ref{assump_1}) of $u$, the NSE (\ref{nse}) become
\begin{align}\label{nse_spec}
	\left\{\begin{matrix}
		&\frac{\partial U}{\partial t}-\nu \Delta U+\frac{\partial P}{\partial x_1}=0, \\ 
		&\frac{\partial P}{\partial x_2}=0, \\ 
		&\frac{\partial P}{\partial x_3}=0,\\
		&\frac{\partial U}{\partial x_1}=0,
	\end{matrix}\right.
\end{align}
hence $U$ is independent of $x_1$. Using the first and fourth equations in (\ref{nse_spec}), we obtain that $\frac{\partial^2 P}{\partial x_1^2}=0$, which combined with the second and third equations in (\ref{nse_spec}) for $P$, imply that $P$ must be of the form
\begin{align}
	\label{form_of_p}
	P=P(x_1,t)=\tilde{p}_0(t)+x_1\tilde{p}_1(t).
\end{align}

Therefore, $U=U(x_2,x_3,t)$ satisfies
\begin{align}
	\label{equ_U}
	\frac{\partial U}{\partial t}-\nu (\frac{\partial^2}{\partial x_2^2}+\frac{\partial ^2}{\partial x_3^2})U=-\tilde{p}_1(t).
\end{align}


\subsubsection{Solving $(\ref{nse_spec})$}
Based on the periodicity (\ref{periodicity_u}) in $(\bf{A}.2)$, we can expand $U(x_2,x_3,t)$ in Fourier series:
\begin{align}
	\label{expansion_U}
	U(x_2,x_3,t)=\sum_{n=-\infty}^{\infty}\hat{U}(n,x_3,t)e^{ \frac{i 2 \pi n}{\Pi_2} x_2},
\end{align}
where 
\begin{align*}
	\hat{U}(n,x_3,t):=\frac{1}{\Pi_2} \int_{0}^{\Pi_2} U(x_2,x_3,t) e^{-\frac{i 2 \pi n}{\Pi_2} x_2} d x_2,
\end{align*}
in particular
\begin{align*}
	\hat{U}(0,x_3,t)=\frac{1}{\Pi_2} \int_{0}^{\Pi_2} U(x_2,x_3,t) d x_2.
\end{align*}

So the equation (\ref{equ_U}) can be written as
\begin{align}
	\label{eqn_in_fourier coeff}
	\sum_{n=-\infty}^{\infty} \frac{\partial}{\partial t} \hat{U}(n,x_3,t)e^{\frac{i 2 \pi n}{\Pi_2} x_2}-\nu \sum_{n=-\infty}^{\infty} \left [-\hat{U}(n,x_3,t) (\frac{2 \pi n}{\Pi_2})^2+(\frac{{\partial}^2}{\partial {x_3}^2})\hat{U}(n,x_3,t)\right ]e^{\frac{i 2 \pi n}{\Pi_2} x_2}=-\tilde{p}_1 (t),
\end{align}
where (\ref{no_slip_bc}) implies the following boundary condition
\begin{align}
	\label{boundary_u_hat}
	\hat{U}(n,x_3,t)|_{x_3=0,h}=0.
\end{align}

We first consider the case when $n \neq 0$:
(\ref{eqn_in_fourier coeff}) implies
\begin{align}\label{nonzero_n}
	\frac{\partial}{\partial t} \hat{U}(n,x_3,t)+\nu (\frac{2 \pi n}{\Pi_2})^2 \hat{U}(n,x_3,t)- \nu (\frac{{\partial}^2}{\partial {x_3}^2}) \hat{U}(n,x_3,t)=0.
\end{align}
Taking dot product with $\hat{U}(n,x_3,t)$ in equation (\ref{nonzero_n}), and then integrating with respect to $x_3$ from $0$ to $h$, we get:
\begin{align*}
	\frac{1}{2} \frac{d}{d t}\int_{0}^{h} |\hat{U}(n,x_3,t)|^2 d x_3+\nu  (\frac{2 \pi n}{\Pi_2})^2 \int_{0}^{h} |\hat{U}(n,x_3,t)|^2 d x_3- \nu \int_{0}^{h}(\frac{{\partial}^2}{\partial {x_3}^2})\hat{U}(n,x_3,t) \cdot \hat{U}(n,x_3,t) d x_3=0.
\end{align*}
Integration by parts, using the boundary condition (\ref{boundary_u_hat}), results in
\begin{align}\label{nonzero_n_integral}
	\frac{1}{2} \frac{d}{d t}\int_{0}^{h} |\hat{U}(n,x_3,t)|^2 d x_3+\nu  (\frac{2 \pi n}{\Pi_2})^2 \int_{0}^{h} |\hat{U}(n,x_3,t)|^2 d x_3+\nu \int_{0}^{h}\left(\frac{{\partial}}{\partial {x_3}}\hat{U}(n,x_3,t)\right)^2 d x_3=0.
\end{align}
By Poincar$\acute{e}$ inequality (\ref{poincare}), we get from equation (\ref{nonzero_n_integral}) 
\begin{align*}
	\frac{1}{2} \frac{d}{d t}\int_{0}^{h} |\hat{U}(n,x_3,t)|^2 d x_3+\nu  (\frac{2 \pi n}{\Pi_2})^2 \int_{0}^{h} |\hat{U}(n,x_3,t)|^2 d x_3+\nu \frac{1}{h^2}\int_{0}^{h} |\hat{U}(n,x_3,t)|^2 d x_3 \leq 0.
\end{align*}
Denoting $W(n,t):=\int_{0}^{h} |\hat{U}(n,x_3,t)|^2 d x_3$, we have
\begin{align*}
	\frac{1}{2} \frac{d}{d t}W(n,t)+\nu  (\frac{2 \pi n}{\Pi_2})^2 W(n,t)+\nu \frac{1}{h^2}W(n,t) \leq 0,
\end{align*}
which can be integrated to get,
\begin{align}
	\label{ineq_for_W}
	W(n,t) \leq e^{-2(t-t_0)(\nu  (\frac{2 \pi n}{\Pi_2})^2+\frac{\nu}{h^2})} W(n,t_0), \forall t_0<t,
\end{align}
then, because of assumption $(\bf{A}.3)$, $W(n,t_0)$ is bounded for all $t_0 \in \mathbb{R}$, hence by taking $t_0\rightarrow -\infty$ in (\ref{ineq_for_W}), we get $W(n,t)=0$, for all $t \in \mathbb{R}$, and therefore, $\hat{U}(n,x_3,t)=0$, for all $n \neq 0$.

So, we only need to consider the case when $n=0$: in this case (\ref{eqn_in_fourier coeff}) implies
\begin{align}
	\label{zero_n}
	\frac{\partial}{\partial t} \hat{U}(0,x_3,t)- \nu (\frac{{\partial}^2}{\partial {x_3}^2})  \hat{U}(0,x_3,t)=-\tilde{p}_1 (t);
\end{align}
and, the expansion of $U(x_2,x_3,t)$ in (\ref{expansion_U}) becomes
\begin{align}\label{tj3}
	U(x_2,x_3,t)=\hat{U}(0,x_3,t)=\frac{1}{\Pi_2} \int_{0}^{\Pi_2} U(x_2,x_3,t) d x_2,
\end{align}
that is, $U(x_2,x_3,t)=U(x_3,t)$ is independent of $x_2$.

The equation satisfied by $U(x_3,t)$ follows from equation (\ref{equ_U}) (or (\ref{zero_n})) :
\begin{align}
	\label{eqn_for_simp_U}
	\frac{\partial}{\partial t}U(x_3,t)-\nu  \frac{{\partial}^2}{\partial {x_3}^2} U(x_3,t)=-\tilde{p}_1 (t).
\end{align}

Since $U(0,t)=U(h,t)=0$, we can take the Fourier sine expansion for $U(x_3,t)$: $U(x_3,t)=\sum_{k=1}^{\infty}\hat{U}(k,t) \sin(\frac{\pi kx_3}{h} )$, so equation (\ref{eqn_for_simp_U}) gives the following equation for the coefficients $\hat{U}(k,t)$
\begin{align}
	\label{eqn_for_coef_simp_U}
	\frac{\partial}{\partial t}\hat{U}(k,t)+\nu  (\frac{\pi k}{h})^2 \hat{U}(k,t) =-\frac{2}{h} \int_{0}^{h} \sin(\frac{\pi kx_3}{h} ) \cdot \tilde{p}_1 (t) d x_3,
\end{align}
thus, we obtain the following explicit form for $U=U(x,t)=U(x_3,t)$.
\begin{theorem}
	Let $u=(U(x,t),0,0)$ be a solution of the NSE $(\ref{nse})$ with $P=P(x,t)$ given, satisfying $(\bf{A}.1)$, $(\bf{A}.2)$ and $(\bf{A}.3)$. Then $U=U(x,t)=U(x_3,t)$ and
	\begin{align}
		\label{form_U(x_3,t)_p(t)}
		U(x_3,t)=\int_{-\infty}^{t}\sum_{k=1}^{\infty}\frac{2 \left( (-1)^k-1 \right )}{k \pi} e^{-\nu (\frac{\pi k}{h})^2 (t-\tau)}\sin(\frac{\pi kx_3}{h})\tilde{p}_1(\tau) d {\tau}   
	\end{align}
\end{theorem}
\begin{proof}
	Indeed,  we see from (\ref{eqn_for_coef_simp_U}) 
	\begin{align*}
		\frac{\partial}{\partial t}\hat{U}(k,t)+\nu  (\frac{\pi k}{h})^2 \hat{U}(k,t) =\frac{2}{k \pi}\tilde{p}_1(\tau) \left ( (-1)^k-1 \right ),
	\end{align*}
	from which, upon integration, we get
	\begin{align*}
		\hat{U}(k,t) =e^{-\nu (\frac{\pi k}{h})^2 (t-t_0)} \hat{U}(k,t_0)+\frac{2}{k \pi} \left( (-1)^k-1 \right )\int_{t_0}^{t}e^{-\nu (\frac{\pi k}{h})^2 (t-\tau)}\tilde{p}_1(\tau) d {\tau},
	\end{align*}
	which, implies (\ref{form_U(x_3,t)_p(t)}), by taking $t_0 \rightarrow -\infty$.
\end{proof}
The above theorem gives that non-stationary solutions for NSE (\ref{nse}) exist, besides those stationary solutions given by (\ref{vel_par_form}) and (\ref{nse_par_form}). In particular, if $\tilde{p}_1(t)$ is a constant, the Poiseuille flow (time independent) is recovered. This matches the form given in $(\ref{nse_par_form})$.

\begin{corollary}
	If, additionally, we assume that $\tilde{p}_1(t)=\tilde{p}_{10}$, where $\tilde{p}_{10}\in \mathbb{R}$ is a constant, then 
	\begin{align}
		\label{explicit_form_U(x_3,t)}
		U(x_3,t)&=\sum_{k=1}^{\infty}\frac{2 \tilde{p}_{10}h^2}{ \nu (k \pi)^3}  \left ( (-1)^k-1 \right) \sin(\frac{\pi kx_3}{h})\\
		&=-\frac{\tilde{p}_{10}}{2\nu}x_3(h-x_3). \nonumber
	\end{align}
\end{corollary}

\subsubsection{A symmetry property}
Consider the following averaged quantity, 
\begin{align}\label{<U>}
	<U>_2(x_3,t)=\frac{1}{\Pi_2}\int_{0}^{\Pi_2}U(x_2,x_3,t)dx_2,
\end{align}
which, from (\ref{no_slip_bc}), satisfies
\begin{align}\label{boundary_<U>}
	<U>_2(x_3,t)|_{x_3=0,h}=0.
\end{align}

We can prove that $<U>_2(x_3,t)$ satisfies the following symmetry property, which is $a$  $priori$ assumed in \cite{CFHO98}-\cite{CFHOTW99} in the study of the steady solutions.
\begin{theorem}
	\label{symmetry_property}
	$<U>_2(x_3,t)$ in $(\ref{<U>})$ satisfies
	\begin{align}\label{sym_<U>}
		<U>_2(h-x_3,t)= <U>_2(x_3,t), 
	\end{align}
	for all $x_3 \in [0,h], t\in \mathbb{R}$.
\end{theorem}

\begin{proof}
	By taking average in $x_2$ as defined in $(\ref{<U>})$ of the equation (\ref{equ_U}), and invoking the periodicity condition (\ref{periodicity_u}), we see that $<U>_2(x_3,t)$ must satisfy
	\begin{align}\label{equ_<U>}
		\frac{\partial <U>_2}{\partial t}-\nu \frac{\partial ^2 <U>_2}{\partial x_3^2}=-\tilde{p}_1(t).
	\end{align}
	
	Define,
	\begin{align}\label{tilde_<U>}
		\widetilde{<U>}(x_3,t):=<U>_2(h-x_3,t)-<U>_2(x_3,t),
	\end{align}
	then 
	\begin{align}\label{equ_tilde_<U>}
		\frac{\partial \widetilde{<U>}}{\partial t}-\nu\frac{\partial ^2 \widetilde{<U>}}{\partial x_3^2}=0,
	\end{align}
	with boundary conditions
	\begin{align}\label{boundary_tilde_<U>}
		\widetilde{<U>}(x_3,t)|_{x_3=0,h}=0,
	\end{align}
	which follows from (\ref{boundary_<U>}).
	
	Multiplying equation (\ref{equ_tilde_<U>}) by $\widetilde{<U>}$ and then integrating with respect to $x_3$ from $0$ to $h$, using the boundary conditions (\ref{boundary_tilde_<U>}), we get
	\begin{align*}
		0&=\int_{0}^{h}\bigg(\frac{\partial \widetilde{<U>}}{\partial t}-\nu\frac{\partial ^2 \widetilde{<U>}}{\partial x_3^2}\bigg) \widetilde{<U>}(x_3,t) dx_3\\
		&=\frac{1}{2}\frac{d}{dt}\int_{0}^{h}\bigg(\widetilde{<U>}(x_3,t)\bigg)^2dx_3-\nu \int_{0}^{h} \frac{\partial ^2 \widetilde{<U>}}{\partial x_3^2}\widetilde{<U>}(x_3,t) dx_3\\
		&=\frac{1}{2}\frac{d}{dt}\int_{0}^{h}\bigg(\widetilde{<U>}(x_3,t)\bigg)^2dx_3+\nu \int_{0}^{h}\bigg(\frac{\partial}{\partial x_3}\widetilde{<U>}(x_3,t)\bigg)^2dx_3.
	\end{align*}
	Invoking Poincar$\acute{e}$ inequality (\ref{poincare}), we obtain
	\begin{align}\label{ineq_tilde_<U>}
		\frac{1}{2}\frac{d}{dt}\int_{0}^{h}\bigg(\widetilde{<U>}(x_3,t)\bigg)^2dx_3+\frac{\nu}{h^2}\int_{0}^{h}\bigg(\widetilde{<U>}(x_3,t)\bigg)^2dx_3 \leq 0,
	\end{align}
	
	therefore, we get from (\ref{ineq_tilde_<U>}) that
	\begin{align}\label{tilde_<U>_sol}
		\psi(t)\leq \psi(t_0)e^{-(t-t_0)2\nu/h^2}, 
	\end{align}
	for $-\infty<t_0<t<\infty$, where $\psi(t):=\int_{0}^{h}\left(\widetilde{<U>}(x_3,t)\right)^2dx_3$ is nonnegative for all $t\in \mathbb{R}$.
	Under the assumption $(\bf{A}.3)$, we could let $t_0\rightarrow -\infty$ in (\ref{tilde_<U>_sol}) to get $\psi(t)\equiv 0, \forall t\in \mathbb{R}$, hence $\widetilde{<U>}(x_3,t)\equiv 0, \forall x_3 \in [0,h], t\in \mathbb{R}$, and consequently, the symmetry property (\ref{sym_<U>}) on $<U>_2(x_3,t)$ is obtained.
\end{proof}

\subsection{The NS-$\alpha$ case}
We assume the following form of solution for the NS-$\alpha$ (\ref{che}):
\begin{align}\label{assump_1_v}
	u=(U(x,t),0,0)
\end{align}
and $V=(1-
\alpha^2\Delta)U$, for $0\leq x_3\leq h$, with $u(x,t)$ satisfying $(\bf{A}.1)$, $(\bf{A}.2)$ and $(\bf{A}.3)$.

\subsubsection{Simple form}
Using (\ref{assump_1_v}), the NS-$\alpha$ (\ref{che}) becomes
\begin{align}\label{che_spec}
	\left\{\begin{matrix}
		&\frac{\partial V}{\partial t}-\nu \Delta V+\frac{\partial Q}{\partial x_1}=0, \\ 
		&V\frac{\partial U}{\partial x_2}=-\frac{\partial Q}{\partial x_2}, \\ 
		&V\frac{\partial U}{\partial x_3}=-\frac{\partial Q}{\partial x_3},\\
		&\frac{\partial U}{\partial x_1}=0.
	\end{matrix}\right.
\end{align}
Using the first and fourth equations in (\ref{che_spec}), we obtain $\frac{\partial ^2 Q}{\partial x_1^2}=0$. By taking partial derivative in the second and third equations in (\ref{che_spec}) with respect to $x_1$, we see $\frac{\partial ^2Q}{\partial x_1 \partial x_2}=0=\frac{\partial ^2Q}{\partial x_3 \partial x_1}$, thus $Q$ must be of the form
\begin{align}\label{form_of_q}
	Q=Q(x_1,x_2,x_3,t)=\tilde{q}_0(x_2,x_3,t)+x_1\tilde{q}_1(t).
\end{align}

Hence, the first equation in (\ref{che_spec}) can be further simplified to be the following equation
\begin{align}\label{equ_V}
	\frac{\partial V}{\partial t}-\nu (\frac{\partial^2}{\partial x_2^2}+\frac{\partial ^2}{\partial x_3^2})V=-\tilde{q}_1(t),
\end{align}
which is strikingly similar to (\ref{equ_U}).

The no-slip boundary condition (\ref{no_slip_bc}) implies
\begin{align}
	\label{no-slip_for_U_CHE}
	U(x_2,x_3,t)|_{x_3=0,h}=0.
\end{align}


\subsubsection{Solving $(\ref{che_spec})$}

By no-slip boundary condition (\ref{no-slip_for_U_CHE}), one could write

$U(x_2,x_3,t)=\sum_{n=-\infty}^{\infty}\hat{U}(n,x_3,t)e^{\frac{i 2 \pi n}{\Pi_2} x_2}$, hence $V=(1-{\alpha}^2 \Delta)U=\left(1+{\alpha}^2 (\frac{2 \pi n}{\Pi_2})^2-{\alpha}^2 \frac{\partial^2}{\partial x_3^2} \right)U$.

Similarly, we have the equations for $\hat{U}(n,x_3,t)$, which follow from (\ref{equ_V}):

(1) when $n=0$
\begin{align*}
	\left [ 1-{\alpha}^2 \frac{\partial^2}{\partial x_3^2} \right ] \left [ \frac{\partial}{\partial t}-\nu \frac{\partial^2}{\partial x_3^2} \right ] \hat{U}(n,x_3,t)=-\tilde{q}_1 (t).
\end{align*}

(2) when $n \neq 0$
\begin{align}\label{tj5}
	\left [ 1+{\alpha}^2 (\frac{2 \pi n}{\Pi_2})^2-{\alpha}^2 \frac{\partial^2}{\partial x_3^2} \right ] \left [ \frac{\partial}{\partial t}+{\nu} (\frac{2 \pi n}{\Pi_2})^2-\nu \frac{\partial^2}{\partial x_3^2} \right ] \hat{U}(n,x_3,t)=0.
\end{align}
Using arguments similar to those given in previous section, we get $\hat{U}(n,x_3,t)=0$, when $n \neq 0$, so $U=\hat{U}(0,x_3,t)=U(x_3,t)$ is also independent of $x_2$, and satisfies,
\begin{align}
	\label{eqn_U_che}
	\left [ 1-{\alpha}^2 \frac{\partial^2}{\partial x_3^2} \right ] \left [ \frac{\partial}{\partial t}-\nu \frac{\partial^2}{\partial x_3^2} \right ] \hat{U}(0,x_3,t)=-\tilde{q}_1 (t).
\end{align}
From (\ref{eqn_U_che}) and (\ref{form_of_p}), we obtain the following theorem
\begin{theorem}
	\label{ge_nse_to_che}
	Let $u=(U(x,t),0,0)$ be a solution of the NSE $(\ref{nse})$ with $P=P(x,t)$ given, satisfying $(\bf{A}.1)$, $(\bf{A}.2)$ and $(\bf{A}.3)$. Then $u$ is also a solution of the NS-$\alpha$ $(\ref{che})$ with $Q=Q(x,t)$ given by,
	\begin{align*}
		Q=Q(x,t)=\tilde{p}_1(t)x_1-\frac{1}{2}\left(U^2-\alpha^2(\frac{\partial U}{\partial x_3})^2\right).
	\end{align*}
\end{theorem}

We express $U=U(x_3,t)=\sum_{k=1}^{\infty}\hat{U}(k,t) \sin(\frac{\pi kx_3}{h})=\sum_{k=1}^{\infty}\hat{U}(k,t) \sin(\frac{\pi k x_3}{h})$. From (\ref{eqn_U_che}), it follows that the Fourier coefficient $\hat{U}(k,t)$ satisfies:
\begin{align}
	\label{U_che_eqn_1}
	\left(1+{\alpha}^2 (\frac{\pi k}{h})^2\right) \left [ \frac{\partial}{\partial t}\hat{U}(k,t)+\nu  (\frac{\pi k}{h})^2 \hat{U}(k,t) \right ] =-\frac{2}{h} \int_{0}^{h} \sin(\frac{\pi kx_3}{h}  ) \cdot \tilde{q}_1 (t) d x_3,
\end{align}
hence, using the same procedure as in previous section, we obtain the following form of the solution $U=U(x_3,t)$,
\begin{theorem}
	\label{che_q(t)}
	The solution $U=U(x_3,t)$ is given by 
	\begin{align}
		\label{form_U(x_3,t)_q(t)}
		U(x_3,t)=\int_{-\infty}^{t}\sum_{k=1}^{\infty}\frac{2\left((-1)^k-1\right)}{\pi k(1+\alpha^2(\pi k/h)^2)}  e^{-\nu(\frac{\pi k}{h})^2(t-\tau)}\sin(\frac{\pi kx_3}{h})\tilde{q}_1(\tau)d\tau.
	\end{align}
\end{theorem}

The above theorem gives that non-stationary solutions for  NS-$\alpha$ (\ref{che}) exist, besides those stationary solutions given by (\ref{vel_par_form}) and (\ref{vche_par_form}). In the particular case when $\tilde{q}_1(t)$ is a constant, we obtain the steady state solution mentioned in \cite{CFHO98}-\cite{CFHOTW99}, which is basically of the form (\ref{vche_par_form}).

\begin{corollary} \footnote{Note that the coefficients in (\ref{U_che_form}) satisfy the conditions (9.7) in \cite{CFHO99} with $c=0$ and $d_0=h/2$, but there are typos in (9.7) there, namely, the left hand sides of the second and the third relations should be, respectively, $\pi_0/{\nu}$ and $\pi_{2}/{\nu}$.}
	\label{che_constant_q}
	If, furthermore, we assume $\tilde{q}_1 (t)=\tilde{q}_{10}$, where $\tilde{q}_{10}\in \mathbb{R}$ is a constant,
	then, 
	\begin{align}
		\label{U_che_form}
		U(x_3,t)&=\sum_{k=1}^{\infty}\frac{2h^2}{\nu( \pi k)^3 } \frac{1}{(1+{\alpha}^2 (\pi k/h)^2) } \tilde{q}_{10} \left((-1)^k-1 \right ) \sin(\frac{\pi kx_3}{h} )\\
		&=\frac{\alpha^2\tilde{q}_{10}}{\nu}\left(1-\frac{\cosh(\frac{x_3-h/2}{\alpha})}{\cosh(\frac{h}{2\alpha})}\right)-\frac{\tilde{q}_{10}}{2\nu}x_3(h-x_3).\nonumber 
	\end{align}
\end{corollary}
\begin{proof}
	(\ref{U_che_form}) follows from using,
	\begin{align*}
		x(h-x)=\sum_{k=1}^{\infty}\frac{4h^2(1-(-1)^k)}{(\pi k)^3}\sin(\frac{\pi k x}{h}),
	\end{align*}
	and
	\begin{align*}
		\cosh(\frac{h}{2\alpha})-\cosh(\frac{x-h/2}{\alpha})=\sum_{k=1}^{\infty}\frac{2\cosh(h/(2\alpha))}{\pi k(1+\alpha^2(\pi k/h)^2)}\left(1-(-1)^k\right)\sin(\frac{\pi k x}{h}).
	\end{align*}
	
\end{proof}

\begin{remark}
	From Corollary $\ref{che_constant_q}$, we see that all solutions having the form $(\ref{assump_1_v})$ of the NS-$\alpha$ $(\ref{che})$ with time independent $\partial Q/{\partial x_1}$ and satisfying $(\bf{A}.1)$, $(\bf{A}.2)$ and $(\bf{A}.3)$ are actually stationary solutions.
\end{remark}

\section{Energy estimate}
In this part, we will use the inequality (\ref{ineq_form}) in Lemma \ref{inequality} to get an inequality estimating the energy of the velocity field $u(x,t)\in \mathcal{P}$.

Taking the dot product of the NSE (\ref{nse}) with $u$ and integrating over $\Omega:=[0,\Pi_1]\times[0,\Pi_2]\times[0,h]$, we get 
\begin{align}\label{form_1_1}
\frac{1}{2}\frac{d}{dt}\int_{\Omega}|u|^2d^3x+\nu \int_{\Omega}|\nabla u|^2d^3x=-\sum_{j=1}^3\int_{\Omega} \frac{\partial P}{\partial x_j}u_jd^3x,
\end{align}
since the nonlinear term $\int_{\Omega}(u\cdot \nabla)u\cdot ud^3x$ vanishes.

Indeed, using integration by parts and the periodicity conditions $(\bf{A}.2)$, one gets,  for $j=1,2$,
\begin{align*}
\int_0^{\Pi_j}u_k\frac{\partial}{\partial x_j}(u_j u_k) dx_j=-\int_0^{\Pi_j}u_ju_k\frac{\partial}{\partial x_j}u_k dx_j,  \forall k=1,2,3;
\end{align*}
similarly, using the boundary condition (\ref{no_slip_bc}) and integration by parts, we obtain
\begin{align*}
\int_0^h u_k\frac{\partial}{\partial x_3}(u_3 u_k) dx_3=-\int_0^h u_3u_k\frac{\partial}{\partial x_3}u_kdx_3, \forall k=1,2,3;
\end{align*}
so,
\begin{align*}
\int_{\Omega}(u\cdot \nabla)u\cdot ud^3x&=\int_{\Omega}\sum_{k=1}^{3} \sum_{j=1}^3 u_j(\frac{\partial}{\partial x_j}u_k)u_kd^3x\\
&=\int_{\Omega}\sum_{k=1}^{3} \sum_{j=1}^3 u_k\frac{\partial}{\partial x_j}(u_ju_k)d^3x\\
&=-\sum_{k=1}^3\int_{\Omega} (u_1u_k\frac{\partial}{\partial x_1}u_k+u_2u_k\frac{\partial}{\partial x_2}u_k+u_3u_k\frac{\partial}{\partial x_3}u_k)d^3x\\
&=-\sum_{k=1}^3\int_{\Omega}u_k\sum_{j=1}^3u_j \frac{\partial}{\partial x_j}u_kd^3x\\
&=-\int_{\Omega}(u\cdot \nabla)u\cdot ud^3x,
\end{align*}
hence,
\begin{align*}
\int_{\Omega}(u\cdot \nabla)u\cdot ud^3x=0.
\end{align*}
\begin{remark}
\label{orthogonality of B}
Observe that the above proof can be applied to show
\begin{align}
\label{ortho_of_B}
\int_{\Omega}(u\cdot \nabla)v\cdot vd^3x=0,
\end{align}
for $u,v\in \mathcal{P}$.
\end{remark}
For the term on the right hand side of (\ref{form_1_1}), we have, by (\ref{pressure_difference}),
\begin{align*}
\int_{\Omega}\frac{\partial P}{\partial x_1}u_1d^3x&=\int_{0}^{\Pi_2} \int_0^h\bigg( (Pu_1)|_{x_1=\Pi_1}-(Pu_1)|_{x_1=0}-\int_0^{\Pi_1}P \frac{\partial u_1}{\partial x_1}dx_1\bigg)dx_3dx_2 \\
&=p_1(t)\int_0^{\Pi_2}\int_0^{h} u_1|_{x_1=0}dx_3dx_2-\int_{\Omega}P\frac{\partial u_1}{\partial x_1}d^3x;
\end{align*}
similarly,
\begin{align*}
\int_{\Omega}\frac{\partial P}{\partial x_2}u_2d^3x=p_2(t)\int_0^{\Pi_1}\int_0^h u_2|_{x_2=0}dx_3dx_1-\int_{\Omega}P \frac{\partial u_2}{\partial x_2}d^3x;
\end{align*}
and, from no-slip boundary condition (\ref{no_slip_bc}),
\begin{align*}
\int_{\Omega}\frac{\partial P}{\partial x_3}u_3d^3x=-\int_{\Omega}P\frac{\partial u_3}{\partial x_3}d^3x;
\end{align*}
so,
\begin{align*}
-\int_{\Omega} \nabla P \cdot u d^3x&=\int_{\Omega}P \nabla \cdot ud^3x-p_1(t)\int_0^{\Pi_2}\int_0^h u_1|_{x_1=0}dx_3dx_2-p_2(t)\int_0^{\Pi_1}\int_0^hu_2|_{x_2=0}dx_3dx_1 \\
&=-p_1(t)\int_0^{\Pi_2}\int_0^h u_1|_{x_1=0}dx_3dx_2-p_2(t)\int_0^{\Pi_1}\int_0^hu_2|_{x_2=0}dx_3dx_1,
\end{align*}
where in the last line, the incompressibility condition (i.e., the second equation in (\ref{nse})) is used.

Therefore, using (\ref{ineq_form}) in Lemma \ref{inequality}, relations (\ref{pressure_difference}), and denoting $j^{\prime}=3-j$ for $j=1,2$, (\ref{form_1_1}) becomes
\begin{align}
\label{for_sigma_att}
\frac{1}{2}\frac{d}{dt}\int_{\Omega}|u|^2d^3x+\nu\int_{\Omega}|\nabla u|^2d^3x&=-p_1(t)\int_0^{\Pi_2}\int_0^h u_1|_{x_1=0}dx_3dx_2-p_2(t)\int_0^{\Pi_1}\int_0^hu_2|_{x_2=0}dx_3dx_1\\ \nonumber
&\leq \sum_{j=1}^2 |p_j(t)|\int_0^{\Pi_{j^{\prime}}}\int_0^h \bigg(<u_j>_j+\frac{\Pi_j}{2\sqrt{3}} <(\frac{\partial u_j}{\partial x_j})^2>_j^{1/2}\bigg)dx_3dx_{{j^{\prime}}}\\ \nonumber
&=\sum_{j=1}^2 |p_j(t)|\int_0^{\Pi_{j^{\prime}}}\int_0^h \bigg(<u_j>_j+\frac{{\Pi_j}^{1/2}}{2\sqrt{3}} \big(\int_0^{\Pi_j}(\frac{\partial u_j}{\partial x_j})^2dx_j\big)^{1/2}\bigg)dx_3dx_{{j^{\prime}}} \nonumber,
\end{align}
where $<\cdot>_j$ denotes the average in the $x_j$ direction, i.e.,
 \begin{align}
\label{avg_in_1v}
<\cdot>_j:=\frac{1}{\Pi_j}\int_{0}^{\Pi_j}\cdot dx_j.
\end{align}
We then use Young's inequality and $(\bf{A}.4)$ to get
\begin{align*}
|p_j(t)|\int_0^{\Pi_{j^{\prime}}} \int_0^h\frac{{\Pi_j}^{1/2}}{2\sqrt{3}} \big(\int_0^{\Pi_j}(\frac{\partial u_j}{\partial x_j})^2dx_j\big)^{1/2}dx_3 dx_{{j^{\prime}}}\leq 
\frac{\bar{p}^2\Pi_1\Pi_2 h}{24 \nu}+\frac{\nu}{2}\int_{\Omega}(\frac{\partial u_j}{\partial x_j})^2d^3x, j=1,2,
\end{align*}
hence,
\begin{align*}
&\frac{d}{dt}\int_{\Omega}|u|^2d^3x+2\nu \int_{\Omega}\sum_{k,l=1}^3(\frac{\partial u_k}{\partial x_l})^2d^3x-\nu \int_{\Omega}\big( (\frac{\partial u_1}{\partial x_1})^2+(\frac{\partial u_2}{\partial x_2})^2\big)d^3x\\
&\leq \frac{\bar{p}^2\Pi_1\Pi_2 h}{6 \nu}+2 \bar{p}\sum_{j=1}^2\int_0^{\Pi_{j^{\prime}}}\int_0^h <u_j>_jdx_3 dx_{j^{\prime}}\\
&\leq \frac{\bar{p}^2\Pi_1\Pi_2 h}{6 \nu}+2\bar{p}\sum_{j=1}^2(\Pi_{j^{\prime}}h)^{1/2}\big(\int_0^{\Pi_{j^{\prime}}}\int_0^h <u_j>_j^2dx_3dx_{j^{\prime}}\big)^{1/2}\\
&\leq  \frac{\bar{p}^2\Pi_1\Pi_2 h}{6 \nu}+\bar{p}^{3/2}(\Pi_1+\Pi_2)h+\bar{p}^{1/2}\sum_{j=1}^2 \int_0^{\Pi_{j^{\prime}}} \int_0^h<u_j>_j^2dx_3dx_{j^{\prime}}.
\end{align*}

To summarize, we obtain,
\begin{proposition}
\label{energy_inequality_statement}
For $u(x,t)\in \mathcal{P}$, 
\begin{align}
\label{energy_inequality}
\frac{d}{dt}\int_{\Omega}|u|^2d^3x+&2\nu \int_{\Omega}\sum_{k,l=1}^3(\frac{\partial u_k}{\partial x_l})^2d^3x-\nu \int_{\Omega}\big( (\frac{\partial u_1}{\partial x_1})^2+(\frac{\partial u_2}{\partial x_2})^2\big)d^3x \nonumber \\
&\leq  \frac{\bar{p}^2\Pi_1\Pi_2 h}{6 \nu}+\bar{p}^{3/2}(\Pi_1+\Pi_2)h+\bar{p}^{1/2}\sum_{j=1}^2 \int_0^{\Pi_{j^{\prime}}} \int_0^h<u_j>_j^2dx_3dx_{j^{\prime}}.
\end{align}

\end{proposition}

\begin{remark}
Following the general procedure as was done in \cite{CP88}, one can start from $(\ref{for_sigma_att})$ and show the existence of the weak global attractor of the NSE $(\ref{nse})$. 
\end{remark}

\section{Harmonicity of $P=P(x,t)$ in the space variables}
\label{section_harm_P}
\subsection{Harmonicity of $P$}
The following property of $P$, namely, harmonicty in the space variable, could be immediately deduced.
Recall that $P=p+\Phi$, where $p$ is the pressure and $\Phi$ is the potential. Notice that from (\ref{nse_simple}), $P$ is independent of $x_3$.
\begin{proposition}
\label{matrix_zero}
Let $u(x,t)\in \mathcal{P}$, then $P=P(x_1,x_2,t)$ is harmonic in the space variables $x_1$ and $x_2$.
\end{proposition}
\begin{proof}
	From (\ref{nse_simple}), by taking ${\partial}/{\partial x_1}$ in the first equation and ${\partial}/{\partial x_2}$ in the second equation and then summing the two resulting equations, we can obtain the following,
	\begin{align}
	\label{operation_one}
	(\frac{\partial u_1}{\partial x_1})^2+2\frac{\partial u_1}{\partial x_2} \frac{\partial u_2}{\partial x_1}+(\frac{\partial u_2}{\partial x_2})^2=-(\frac{\partial^2}{\partial x_1^2}+\frac{\partial^2}{\partial x_2^2})P,
	\end{align}
	where $P=P(x_1,x_2,t)$ is independent of $x_3$ (see the third equation in (\ref{nse_simple})). 
	
	In (\ref{operation_one}), the left hand side (LHS) takes values zero at $x_3=0$ and $x_3=h$, while the right hand side (RHS) is independent of $x_3$, hence
	\begin{align}
	\label{harmonic_p}
	-(\frac{\partial^2}{\partial x_1^2}+\frac{\partial^2}{\partial x_2^2})P=0, 
	\end{align}
	for all $x_1,x_2 \in \mathbb{R}$.
\end{proof}
An intriguing corollary of the harmonicity of $P$ is the following relation.

\begin{corollary}
	\begin{align}
	\label{jacob_zero}
	\frac{\partial(u_1,u_2)}{\partial(x_1,x_2)}=\begin{vmatrix}
	\frac{\partial u_1}{\partial x_1}& \frac{\partial u_1}{\partial x_2} \\ 
	\frac{\partial u_2}{\partial x_1}&  \frac{\partial u_2}{\partial x_2}
	\end{vmatrix}=0
	\end{align}
\end{corollary}
\begin{proof}
	(\ref{jacob_zero}) is obtained by using the fourth relation in (\ref{nse_simple}), and the fact that the LHS in (\ref{operation_one}) equals zero.
\end{proof}

\begin{remark}
	\label{functional_rel}
	For $u=(u_1,u_2,u_3)$ being the velocity field for channel flows, after we assume $u_3=0$. Corollary $\ref{jacob_zero}$ tells us that the two nonzero components, $u_1$ and $u_2$,  are not totally independent, one of them is, at least locally, a function of the other component. 
\end{remark}
\subsection{An estimate of $P$}

Proving he following estimate for $P$ is elementary.
\begin{lemma}
	\label{poisson_lem}
	For the term $P=P(x_1,x_2,t)$, we have,
	\begin{align}
	\label{bound_p}
	\sup_{x_1,x_2}|\frac{\partial}{\partial x_1}P(x_1,x_2,t)|<\infty,
	\end{align}
	for all $t\in \mathbb{R}$.
\end{lemma}

\begin{proof}
	According to Poisson's formula (see \cite{JF82}; see also \cite{HK07}), we have, for any $a>0, a\in \mathbb{R}$,
	\begin{align*}
	P(z,t)=\int_{|y|=a}H(y,z)P(y,t)dy,  \text{ for } |z|<a,
	\end{align*}
	where $z=(x_1,x_2)$, and
	\begin{align*}
	H(y,z)=\frac{1}{2\pi a}\frac{a^2-|z|^2}{|z-y|^2}.
	\end{align*}

	So,
	\begin{align}
	\label{2a_pressure}
	P(z,t)=P(x_1,x_2,t)&=\frac{1}{2\pi} \int_{|y|=1}\frac{1-|\frac{z}{a}|^2}{|\frac{z}{a}-y|^2}P(ay_1,ay_2,t)dy\\
	&=\frac{1}{2\pi}\int_0^{2\pi} \frac{1-\frac{x_1^2+x_2^2}{a^2}}{1-2(\frac{x_1^2+x_2^2}{a^2})^{1/2}\cos(\theta-\omega)+\frac{x_1^2+x_2^2}{a^2}}P(a\cos\theta,a\sin\theta,t)d\theta, \nonumber
	\end{align}
	where $x_1+ix_2=|z|e^{i\omega}$ and $y_1+iy_2=|y|e^{i\theta}$.
	
	However, we will work with the following equivalent form of (\ref{2a_pressure}), namely,
	\begin{align}
	\label{work_p}
	P(x_1,x_2,t)=\frac{1}{2\pi}\int_0^{2\pi} H(z/a,\theta)P(a\cos\theta,a \sin\theta,t)d\theta,
	\end{align}
	where, for $|z|<1$,
	\begin{align*}
	H(z,\theta)&=\frac{1-x_1^2-x_2^2}{(x_1-\cos\theta)^2+(x_2-\sin\theta)^2}\\
	&=\frac{1-x_1^2-x_2^2}{1+x_1^2+x_2^2-2(x_1\cos\theta+x_2 \sin\theta)}.
	\end{align*}
	A direct calculation gives that
	\begin{align*}
	\frac{\partial}{\partial x_1} H(x_1,x_2,\theta)=\frac{-4x_1+2\cos\theta+2x_1^2\cos\theta-2x_2x_2^2\cos\theta+4x_1x_2\sin\theta}{[1+x_1^2+x_2^2-2(x_1\cos\theta+x_2 \sin\theta)]^2},
	\end{align*}
	this implies, for $a>4|z|+1$, the following
	\begin{align*}
	\frac{\partial}{\partial x_1} H(z/a,\theta)&=\frac{1}{a}\frac{\partial H}{\partial x_1}(z,\theta)|_{z=z/a}\\
	&\leq  \frac{4(\frac{|x_1|}{a}+2\frac{|z|^2}{a^2}+1)}{a(1+\frac{|z|^2}{a^2}-2\frac{|z|}{a})^2}\\
	&\leq 32.
	\end{align*}

	Therefore,  from (\ref{work_p}), recalling the bound (\ref{useful_b_p}) given in Lemma \ref{bound_pressure}, we have,
	\begin{align*}
	|\frac{\partial}{\partial x_1} P(x_1,x_2,t)|\leq \frac{32\bar{p}}{\Pi_1}+32 a \sup\{|P(x_1,x_2,t)|:0 \leq x_1<\Pi_1\},
	\end{align*}
	where $\sup\{|P(x_1,x_2,t)|:0 \leq x_1<\Pi_1\}$ is a periodic function in $x_2$ with period $\Pi_2$, and hence
	\begin{align*}
	\max_{x_2\in \mathbb{R}}\sup\{|P(x_1,x_2,t)|:0 \leq x_1<\Pi_1\}<\infty,
	\end{align*}
	consequently, 
	\begin{align*}
	\sup_{x_1,x_2}|\frac{\partial}{\partial x_1}P(x_1,x_2,t)|<\infty,
	\end{align*}
	for all $t\in \mathbb{R}$.
\end{proof}

\subsection{Simple form of $P$}
From (\ref{pressure_difference}), we see $P(x_1+n\Pi_1,x_2,t)-P(x_1,x_2,t)=np_1(t), \forall n\in \mathbb{N}^{+}$. Now, for any $y\in \mathbb{R}^{+}$, choose $n\in \mathbb{N}$, such that $n\Pi_1\leq y < (n+1)\Pi_1$, then
\begin{align*}
P(y,x_2,t)&=P(y-n\Pi_1,x_2,t)+np_1(t)\\
&\leq \sup\{|P(x_1,x_2,t)|:0\leq x_1<\Pi_1\}+np_1(t)\\
&\leq \sup\{|P(x_1,x_2,t)|:0\leq x_1<\Pi_1\}+\frac{y}{\Pi_1}\bar{p};
\end{align*}
similar arguments apply to the case when $y\leq 0$, and we get the next result.

\begin{lemma}
\label{bound_pressure}
Let $u(x,t)\in \mathcal{P}$ be a solution of $(\ref{nse})$. The estimate
\begin{align}
\label{useful_b_p}
P(y,x_2,t)\leq \sup\{|P(x_1,x_2,t)|:0\leq x_1<\Pi_1\}+\frac{|y|}{\Pi_1}\bar{p}
\end{align}
holds for all $y,x_2,t\in \mathbb{R}$, where $\bar{p}$ is as given in $(\bf{A}.4)$.
\end{lemma}

Next, we will explore the harmonicity of $P=P(x,t)=P(x_1,x_2,t)$ in $x_1$ and $x_2$, and get the following explicit formula for $P(x,t)$, namely, $P(x,t)$ is linear in the variable $x_1$.

\begin{proposition}
\label{Liouville}
The term $P=P(x_1,x_2,t)$ in $(\ref{nse_simple})$ is of the following form,
\begin{align}
\label{form_p}
P(x_1,x_2,t)&=\tilde{p}_0(t)+x_1\tilde{p}_1(t)\\
&=\tilde{p}_0(t)+x_1p_1(t)/\Pi_1. \nonumber
\end{align}
\end{proposition}

\begin{proof}
Using (\ref{bound_p}) in Lemma \ref{poisson_lem} given in Appendix C, and Liouville's theorem for harmonic function $\frac{\partial}{\partial x_1}P(x_1,x_2,t)$, we conclude that
\begin{align*}
\frac{\partial}{\partial x_1}P(x_1,x_2,t)=\tilde{p}_1(t),
\end{align*}
for some function $\tilde{p}_1(t)$ of time $t$, so that
\begin{align}
\label{temp_eqn_p}
P=P(x_1,x_2,t)=\tilde{p}_0(x_2,t)+x_1\tilde{p}_1(t),
\end{align}
and, due to the harmonicity of $P(x_1,x_2,t)$ in $x_1$ and $x_2$,
\begin{align*}
\frac{\partial^2}{\partial x_2^2}\tilde{p}_0(x_2,t)=0,
\end{align*}
so, $\tilde{p}_0(x_2,t)$ is a linear function in $x_2$, but then periodicity of $P(x_1,x_2,t)$ in $x_2$ would imply that $\tilde{p}_0(x_2,t)$ is only a function of time $t$, say 
\begin{align*}
\tilde{p}_0(x_2,t)=\tilde{p}_0(t).
\end{align*}

Finally, the second equality in (\ref{form_p}), namely, 
\begin{align}
\label{two_ps}
\Pi_1 \tilde{p}_{1}(t)=p_1(t),
\end{align}
follows from (\ref{temp_eqn_p}) and relation (\ref{pressure_difference}).
\end{proof}

It follows from the harmonicity of $P$ and Proposition \ref{Liouville} that we can replace $(\bf{A.}5)$ by a weaker condition, namely,
\begin{corollary}
\label{rk_A5}
$(\bf{A.}5)$ can be replaced by the following equivalent weaker condition
\begin{align}
\label{wk_A5}
\limsup_{x_2\rightarrow \pm \infty}P(x_1,x_2,x_3,t)<\infty,
\end{align}
for any given $x_1,x_3$ and $t\in \mathbb{R}$.
\end{corollary}




\section{Appendices}

\subsection{Appendix A: basic inequalities}

In this appendix, we include several classic inequalities that are used in our discussion, together with their proofs.

\subsubsection{Poincar$\acute{e}$ inequality}
\begin{lemma}[Poincar$\acute{e}$ inequality]
\label{poincare}
For any $C^1$ function $\phi(y)$ defined on $[0,h]$, with $\phi(0)=\phi(h)=0$, we have
\begin{align}\label{poincare ineq}
\int_{0}^{h} (\phi^{\prime}(y))^2dy\geq \frac{1}{h^2} \int_0^h(\phi(y))^2dy.
\end{align}
\end{lemma}

\begin{proof}[Proof]
By the fundamental theorem of calculus, we have, for any $x\in \mathbb{R}$, and $x\in[0,h]$, 
\begin{align}\label{poin_one}
\phi^2(x)=2\int_{0}^{x}\phi(y)\phi^{\prime}(y)dy,
\end{align}
and,
\begin{align}\label{poin_two}
\phi^2(x)=-2\int_{x}^{h}\phi(y)\phi^{\prime}(y)dy.
\end{align}
Using Cauchy inequality in (\ref{poin_one}), we get,
\begin{align*}
\phi^2(x)\leq 2\bigg(\int_{0}^{x}\phi^2(y)dy\bigg)^{1/2}\bigg(\int_{0}^{x}(\phi^{\prime}(y))^2dy\bigg)^{1/2},
\end{align*}
hence
\begin{align*}
\int_{0}^{h/2} \phi^2(x)dx&\leq 2\int_{0}^{h/2}\bigg(\int_{0}^{x}\phi^2(y)dy\bigg)^{1/2}\bigg(\int_{0}^{x}(\phi^{\prime}(y))^2dy\bigg)^{1/2}dx\\
&\leq 2\int_{0}^{h/2}\bigg(\int_{0}^{h/2}\phi^2(y)dy\bigg)^{1/2}\bigg(\int_{0}^{h/2}(\phi^{\prime}(y))^2dy\bigg)^{1/2}dx\\
&\leq h \bigg(\int_{0}^{h/2}\phi^2(y)dy\bigg)^{1/2}\bigg(\int_{0}^{h/2}(\phi^{\prime}(y))^2dy\bigg)^{1/2},
\end{align*}
that is, 
\begin{align}\label{poin_3}
\int_{0}^{h/2} \phi^2(x)dx\leq h^2 \int_{0}^{h/2}(\phi^{\prime}(y))^2dy.
\end{align}
Similarly, using (\ref{poin_two}), and repeating the above steps, we get,

\begin{align}\label{poin_4}
\int_{h/2}^{h} \phi^2(x)dx\leq h^2 \int_{h/2}^{h}(\phi^{\prime}(y))^2dy.
\end{align}
Combined (\ref{poin_3}) and (\ref{poin_4}), we obtain (\ref{poincare ineq}).
\end{proof}

\subsubsection {$L^{\infty}$ inequality}
Recall that, in our paper, for a given function $\phi=\phi(y)$ with periodicity $\Pi>0$, we denote its average by $<\phi>$, i.e.,
\begin{align*}
<\phi>:=\frac{1}{\Pi}\int_{0}^{\Pi}\phi(y)dy.
\end{align*}

\begin{lemma} \label{inequality}
 For any continuous function $\phi=\phi(y)$ with periodicity $\Pi>0$, it holds that
\begin{align}\label{ineq_form}
|\phi|_{L^{\infty}}\leq <\phi>+\frac{\Pi}{2\sqrt{3}}<(\phi^{\prime})^2>^{1/2}.
\end{align}
Consequently, if 
\begin{align*}
<(\phi^{\prime})^2><\infty,
\end{align*}
then $\phi$ is continuous in $\mathbb{R}$, and thus $|\phi(y)|\leq |\phi|_{L^{\infty}}, \forall y \in \mathbb{R}$.
\end{lemma}

\begin{proof}[Proof]

Without loss of generality, we assume $\phi(0)=|\phi|_{L^{\infty}}$, then
\begin{align*}
\phi(0)\leq \left\{\begin{matrix}
\phi(y)+\int_0^y|\phi^{\prime}(z)|dz, y\geq 0;\\ 
\phi(y)+\int_y^0|\phi^{\prime}(z)|dz, y\leq 0;
\end{matrix}\right.
\end{align*}
thus,
\begin{align*}
\frac{\Pi}{2}\phi(0) & \leq  \left\{\begin{matrix}
\int_{0}^{\Pi/2}\phi(y)dy+\int_0^{\Pi/2}(\int_z^{\Pi/2}dy)|\phi^{\prime}(z)|dz\\ 
\int_{-\Pi/2}^0\phi(y)dy+\int_{-\Pi/2}^0(\int_{-\Pi/2}^z dy)|\phi^{\prime}(z)|dz
\end{matrix}\right. \\
&=\left\{\begin{matrix}
\int_0^{\Pi/2}\phi(y)dy+\int_0^{\Pi/2}(\Pi/2-z)|\phi^{\prime}(z)|dz;\\ 
\int_{-\Pi/2}^0\phi(y)dy+\int_{-\Pi/2}^0(\Pi/2+z)|\phi^{\prime}(z)|dz;
\end{matrix}\right.
\end{align*}
hence,
\begin{align*}
\Pi |\phi|_{L^{\infty}}=\Pi \phi(0) &\leq \Pi<\phi>+\bigg(\int_0^{\Pi/2}(\Pi/2-z)^2dz\bigg)^{1/2}\bigg(\int_0^{\Pi/2}(\phi^{\prime}(z))^2dz\bigg)^{1/2}\\
&+\bigg(\int_{-\Pi/2}^0 (\Pi/2+z)^2dz\bigg)^{1/2}\bigg(\int_{-\Pi/2}^0 (\phi^{\prime}(z))^2dz\bigg)^{1/2} \\
&\leq \Pi <\phi>+\frac{\Pi^2}{2\sqrt{3}}<(\phi^{\prime})^2>^{1/2},
\end{align*}
and (\ref{ineq_form}) follows.
\end{proof}






\section{Acknowledgement}
This work was supported in part by NSF grants number DMS-1109638 and DMS-1109784. The authors would like to thank Professor C. Foias for suggesting the problem and for subsequent useful discussions and remarks.


\end{document}